\documentclass[11pt]{article}

\usepackage{amssymb}
\usepackage{amsmath}
\usepackage{amsthm}
\usepackage{enumitem}
\usepackage{graphicx}
\usepackage{subfigure}
\usepackage{float}

\newtheoremstyle{plain}%
    {8pt plus2pt minus4pt}%
    {8pt plus2pt minus4pt}%
    {\itshape}%
    {}%
    {\bfseries\scshape}%
    {}%
    {6pt}% Space after theorem head
    {}%

\newtheoremstyle{remark}%
    {8pt plus2pt minus4pt}%
    {8pt plus2pt minus4pt}%
    {\upshape}% Body font
    {}%
    {\bfseries\scshape}%
    {}%
    {6pt}% Space after theorem head
    {}%

\theoremstyle{plain}
\newtheorem{thm}{Theorem}[section]
\newtheorem{cor}[thm]{Corollary}

\newtheorem{lem}[thm]{Lemma}

\newtheorem*{c1}{\textbf{Chernoff Bound}}
\newtheorem*{ub}{\textbf{Union Bound}}
\newtheorem*{lovasz}{\textbf{Lov\'asz Local Lemma}}
\newtheorem{que}[thm]{Question}
\newtheorem{obs}[thm]{Observation}

\theoremstyle{remark}

\title{On generalized Ramsey numbers of Erd\H{o}s and Rogers}

\author{{{Andrzej Dudek\thanks{Supported in part by Simons Foundation Grant \#244712.}}}\\
\small Department of Mathematics\\[-0.8ex]
\small Western Michigan University\\[-0.8ex]
\small Kalamazoo, MI\\
\small \texttt{andrzej.dudek@wmich.edu}
\and
{{Troy Retter}}\\
\small Department of Mathematics and\\[-1ex] 
\small Computer Science\\[-0.8ex]
\small Emory University\\[-0.8ex]
\small Atlanta, GA\\
\small \texttt{tretter@emory.edu}
\and
{{Vojt\v{e}ch R\"odl\thanks{Supported in part by NSF grant DMS 0800070.}}}\\
\small Department of Mathematics and\\[-1ex] 
\small Computer Science\\[-0.8ex]
\small Emory University\\[-0.8ex]
\small Atlanta, GA\\
\small \texttt{rodl@mathcs.emory.edu}}

\begin{document}
\maketitle		
		
\begin{abstract}
Extending the concept of Ramsey numbers, Erd{\H o}s and Rogers introduced the following function. For given integers $2\le s<t$ let
$$
f_{s,t}(n)=\min \{ \max \{ |W| : W\subseteq V(G) \mbox{ and } G[W] \mbox{ contains no } K_s\} \},
$$
where the minimum is taken over all $K_t$-free graphs~$G$ of order~$n$. In this paper, we show that for every $s\ge 3$ there exist constants $c_1=c_1(s)$ and $c_2=c_2(s)$ such that $f_{s,s+1}(n) \le c_1 (\log n)^{c_2} \sqrt{n}$. This result is best possible up to a polylogarithmic factor. We also show for all $t-2 \geq s \geq 4$, there exists a constant $c_3$ such that $f_{s,t}(n) \le c_3 \sqrt{n}$. In doing so, we partially answer a question of Erd\H{o}s by showing that $\lim_{n\to \infty} \frac{f_{s+1,s+2}(n)}{f_{s,s+2}(n)}=\infty$ for any $s\ge 4$.
\end{abstract}

%%%%%%%%%%%%%%%%%%%%%%%%%%%%%%%%%%%%%%%%%%%%%%%%%%%%%%%%%%%%
%%%%%%%%%%%%%%%%%%%%  INTRODUCTION   %%%%%%%%%%%%%%%%%%%%%%%
%%%%%%%%%%%%%%%%%%%%%%%%%%%%%%%%%%%%%%%%%%%%%%%%%%%%%%%%%%%%
\section{Introduction}
In a graph $G$, a set $S\subseteq V(G)$ is \emph{independent} if $G[S]$ does not contain a copy of $K_2$. More generally for any integer $s$, a set $S \subseteq V(G)$ can be called $s$\emph{-independent} if $G[S]$ does not contain a copy of $K_s$. With this in mind, define the $s$\emph{-independence number} of $G$, denoted by $\alpha_s(G)$, to be the size of the largest $s$-independent set in $G$. The classical Ramsey number $R(t,u)$ can be defined in this language as the smallest integer~$n$ such that every graph of order~$n$ contains either a copy of $K_t$ or a 2-independent set of size $u$. In other words, $R(t,u)$ is the smallest integer $n$ such that 
$$
u \leq \min \{\alpha_2(G): G \text{ is a } K_t \text{-free graph of order } n \}.
$$ 
Observe that if the right hand side of the above inequality is understood as a function of $n$ and $t$, then so is the classical Ramsey number.

A more general problem results by replacing the standard independence number by the $s$-independence number for some $2\le s<t$. Following this approach, in 1962 Erd{\H o}s and Rogers~\cite{ERog} introduced the function 
$$
f_{s,t}(n)=\min \{\alpha_s(G): G \text{ is a } K_t \text{-free graph of order } n \}.
$$
\noindent The lower bound $k \le f_{s,t}(n)$ means that every $K_t$-free graph of order~$n$ contains a subset of $k$ vertices with no copy of~$K_s$. The upper bound $f_{s,t}(n) < \ell$ means that there exists a $K_t$-free graph of order~$n$ such that every subset of $\ell$ vertices contains a copy of~$K_s$.

The case $t = s+1$ has received the considerable attention over the last 50 years, in part due to the fact that it creates a general upper bound in the sense that for $t' > t$, we clearly have $f_{s,t'}(n) \leq f_{s,t}(n)$. A first nontrivial upper bound for $f_{s,s+1}(n)$ was established by Erd{\H o}s and Rogers~\cite{ERog}, which was subsequently addressed by Bollob\'as and Hind~\cite{BH}, Krivelevich~\cite{KR2, KR}, Alon and Krivelevich \cite{AK}, Dudek and R\"odl \cite{DR}, and most recently Wolfovitz~\cite{Wo}. The first nontrivial lower bound established by Bollob\'as and Hind~\cite{BH} was later slightly improved by Krivelevich~\cite{KR}. The most recent general bounds for $s\geq 3 $ were of the form:
\begin{equation}\label{eq:bound_s}
\Omega\left(\sqrt{\frac{n \log n }{\log\log n}} \right) = f_{s, s+1}(n) = O\big{(}n^{\frac{2}{3}} \big{)}.
\end{equation}
The precise lower bound of \eqref{eq:bound_s} was first explicitly stated by Dudek and Mubayi~\cite{DM}, and was based upon their observation that the result of Krivelevich \cite{KR} could be slightly strengthened by incorporating a result of Shearer \cite{SH}. This upper bound of \eqref{eq:bound_s} appears in \cite{DR}, where it was also conjectured that for all sufficiently large $s$ the upper bound could be improved to show that
\begin{equation}\label{conj:1}
f_{s,s+1}(n) = n^{\frac{1}{2}+o(1)}.
\end{equation}
Recently, Wolfovitz~\cite{Wo} showed that \eqref{conj:1} holds when $s=3$. In this paper, \eqref{conj:1} is proved for every $s \geq 3$, establishing an upper bound that is tight up to a polylogarithmic factor. Our proof builds upon the ideas in \cite{Wo}, \cite{DR}, \cite{KR2}, and \cite{KR}.

\begin{thm}\label{thm:s1}
For every $s \geq 3$ there is a constant $c=c(s)$ such that 
$$
f_{s,s+1}(n) \le c (\log n)^{4s^2}\sqrt{n}.
$$
\end{thm}
\noindent

For the case $t = s+2$, it follows from a result of Sudakov~\cite{SU} (see also \cite{DR} for a simplified formula) that $f_{s,s+2}(n) = \Omega(n^{a_2})$, where $\frac{1}{a_2} = 2+\frac{2}{3s-4}$. On the other hand, clearly $f_{s,s+2}(n) \leq f_{s,s+1}(n)$. When $s \geq 4,$ we establish an improved upper bound that omits the logarithmic factor.
\begin{thm}\label{thm:s2}
For every $s \geq 4$ there is a constant $c=c(s)$ such that
$$
f_{s,s+2}(n) \le c\sqrt{n}.
$$
\end{thm}
\noindent
This establishes the following corollary which provides the best known bounds on $f_{s,t}(n)$ for $t < 2s$.
\begin{cor} \label{cor3}
For every $6 \leq s +2 \leq t$ there is a constant $c=c(s)$ such that
$$
f_{s,t}(n) \le f_{s,s+2}(n) \le c\sqrt{n}.
$$
\end{cor}
\noindent
When $t \geq 2s$, the upper bound $c(\log n)^{1/(s-1)} n^{s/(t+1)}$ of Krivelevich~\cite{KR2} remains best. For all values of $t>s+1$, the best lower bounds follow from a recursive formula defined by Sudakov~\cite{SU, SU2}. We will return to the these results concerning the general case in our concluding remarks. More related results are summarized in the survey~\cite{DR2}. 

Now that our two main results have been stated, we turn our attention towards an old question of Erd\H{o}s~\cite{ER}, who asked if for fixed integers $s+1 < t$,
\begin{equation}\label{eq:erdos}
\lim_{n\to\infty} \frac{f_{s+1,t}(n)}{f_{s,t}(n)} = \infty.
\end{equation}
This central conjecture in the area is still wide open and asks for a rather precise estimation of $f_{s,t}(n)$. It is known due to Sudakov~\cite{SU2} that \eqref{eq:erdos} holds for 
$$
(s,t)\in\{ (2,4), (2,5), (2,6), (2,7), (2,8), (3,6)\}.
$$
Observe that Theorem~\ref{thm:s2} together with the lower bound of \cite{KR} (and \cite{DR}) implies that for $s\ge 4$,
$$
 \frac{f_{s+1,s+2}(n)}{f_{s,s+2}(n)} \ge \frac{\Omega\left(\sqrt{\frac{n \log n }{\log\log n}} \right)}{O(\sqrt{n})}
= \Omega\left(\sqrt{\frac{\log n }{\log\log n}}\right)
\xrightarrow[n\to\infty]{}\infty.
$$
That is, \eqref{eq:erdos} holds for all pairs $(s,t) \in\{ (4,6), (5,7), (6,8), \dots \}$.

In what follows, consider $s$ to be an arbitrary fixed integer and $n$ sufficiently large, i.e. $n \geq n_0(s)$. We will show that there exists a $K_{s+1}$-free graph of order~$n$ such that every subset of $c (\log n)^{4s^2}\sqrt{n}$ vertices contains a copy of~$K_s$ and that there exists a $K_{s+2}$-free graph of order~$n$ such that every subset of $c \sqrt{n}$ vertices contains a copy of~$K_s$. Indeed, this establishes Theorems \ref{thm:s1} and \ref{thm:s2} as stated (for all $n$), since the constant factors can subsequently be inflated to accommodate the finitely many cases where $n \leq n_0$. For simplicity, we do not round numbers that are supposed to be integers either up or down; this is justified since these rounding errors are negligible to the asymptomatic calculations we will make.

In Section~\ref{sec:H}, we begin our construction by considering the random hypergraph $\mathbb{H}$ which is essentially the random hypergraph obtained from the affine plane by taking each hyperedge (line) with some uniform probability. We then use $\mathbb{H}$ in Section~\ref{sec:G} to construct a random graph $\mathbb{G}$ by replacing each hyperedge by a complete $s$-partite graph. In Section \ref{sec:proofs}, the proof of Theorem \ref{thm:s2} considers an induced subgraph of $\mathbb{G}$ whereas the proof of Theorem \ref{thm:s1} considers yet another random subgraph of $\mathbb{G}$ which is analyzed by way of the Local Lemma.

Below we will use the standard notation to denote the neighborhood and degree of $v \in G$ by $N_G(v)$ and $d_G(v)$ respectively.

%%%%%%%%%%%%%%%%%%%%%%%%%%%%%%%%%%%%%%%%%%%%%%%%%%%%%%%%%%%%
%%%%%%%%%%%%%%%%%%%%  THE HYPERGRAPH H  %%%%%%%%%%%%%%%%%%%%
%%%%%%%%%%%%%%%%%%%%%%%%%%%%%%%%%%%%%%%%%%%%%%%%%%%%%%%%%%%%

\section{The Hypergraph $H$} \label{sec:H}

The {\em affine plane} of order $q$ is an incidence structure on a set of $q^2$ points and a set of $q^2+q$ lines such that: any two points lie on a unique line; every line contains~$q$ points; and every point lies on~$q+1$ lines. It is well known that affine planes exist for all prime power orders.
Clearly, an incidence structure can be viewed as a hypergraph with points corresponding to vertices and lines corresponding to hyperedges; we will use this terminology interchangeably.

In the affine plane, call lines $L$ and $L'$ \emph{parallel} if $L \cap L' = \emptyset$. In the affine plane there exist $q+1$ sets of $q$ pairwise parallel lines. (For more details see, e.g., \cite{CA}.) Let $(V, \mathcal{L})$ be the hypergraph obtained by removing a parallel class of $q$ lines from the affine plane or order~$q$. The following lemma establishes some properties of this graph.

\begin{lem} \label{lem:0}
For $q$ prime, the $q$-uniform, $q$-regular hypergraph $(V,\mathcal{L})$ of order $q^2$ satisfies:
\begin{enumerate}[label=(P\arabic*)]
\item\label{P:1} Any two vertices are contained in at most one hyperedge;
\item\label{P:2} For every $A \in {V \choose q}$, $|\{L \in \mathcal{L} : L \cap A \not = \emptyset \}| \geq \frac{q^2}{2}$.
\end{enumerate}
\end{lem}

\begin{proof}
By construction, $(V,\mathcal{L})$ is $q$-uniform, $q$-regular, and satisfies \ref{P:1}. Consider $A=\{v_1,v_2,\dots ,v_q\}$. Define $d_{+}(v_i) = |\{ L \in \mathcal{L} : L \cap \{v_1,v_2,\dots, v_{i}\} = \{v_i\} \}|$. Hence by property~\ref{P:1}, $d_+(v_i) \geq q-i+1$. We now compute
$$ |\{L \in \mathcal{L} : L \cap A \not = \emptyset \}| \geq \sum_{i=1}^q d_+(v_i) \geq {q+1 \choose 2} \geq \frac{q^2}{2}.$$
\end{proof}

The objective of this section is to establish the existence of a certain hypergraph $(V,\mathcal{L}') \subseteq (V,\mathcal{L})$ by considering a random sub-hypergraph of $(V,\mathcal{L})$. Preceding this, we introduce some terminology. Define 
$$\mathcal{L}_A' = \{L \in \mathcal{L}' : L \cap A \not = \emptyset \}, \hspace{1cm} \text{ and } \hspace{1cm} \mathcal{L}'_{B, \gamma } = \{L \in \mathcal{L}' : |L \cap B| \geq \gamma \}.$$ 
Call $S \subseteq V$ $\mathcal{L}'$-\emph{complete} if every pair of points in $S$ is contained in some common line in $\mathcal{L}'$. Let $L(x,y)$ denote the unique line in $\mathcal{L}$ containing $x$ and $y$, provided such a line exists. 

We will now distinguish 3 types of \emph{$\mathcal{L}'$-dangerous} subsets as depicted in Figure \ref{fig:ks2}. The first two types have 5 vertices $\{v_1,v_2,v_3,v_4,x\}$ and third type has 6 vertices $\{v_1,v_2,v_3,v_4,y,z\}$. All 3 types of dangerous sets must be $\mathcal{L}'$-complete and have 4 points $\{v_1,v_2,v_3,v_4\}$ in general position. Additionally we specify:
\begin{description}
  \item[Type 1 $\mathcal{L}'$-dangerous] \hfill \\
  The points $\{v_1,v_2,v_3,v_4,x\}$ are in general position.
  \item[Type 2 $\mathcal{L}'$-dangerous] \hfill \\
  The point $x$ is contained in precisely one of the 6 lines $L(v_i,v_j)$ for $1 \leq i < j \leq 4$. Up to relabeling, say $x \in L(v_2,v_3)$.
  \item[Type 3 $\mathcal{L}'$-dangerous] \hfill \\
  The points $y$ and $z$ are each contained in exactly two of the lines $L(v_i,v_j)$ for $1 \leq i < j \leq 4$. Up to relabeling, say $y \in L(v_1,v_3) \cap L(v_2,v_4)$ and $z \in L(v_1,v_2) \cap L(v_3,v_4)$.
\end{description}
All concepts above were defined relative to the subset $\mathcal{L}' \subseteq \mathcal{L}$. Obviously we can define the concepts $\mathcal{L}$-complete, $\mathcal{L}$-dangerous, $\mathcal{L}_A$, and $\mathcal{L}_{B, \gamma }$ related to the set $\mathcal{L}$ analogously.

\begin{figure}
\begin{center}
\subfigure[Type 1]{\includegraphics[scale=0.4]{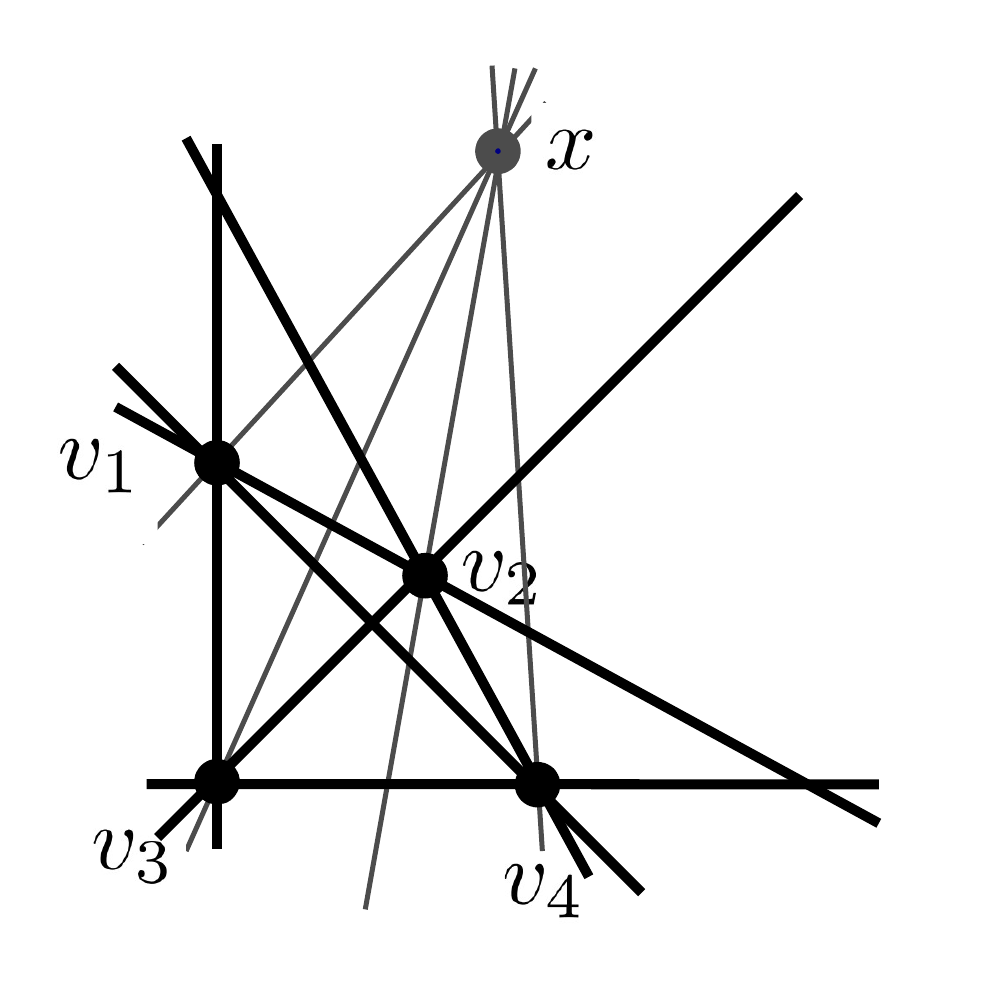}}
\qquad
\subfigure[Type 2]{\includegraphics[scale=0.4]{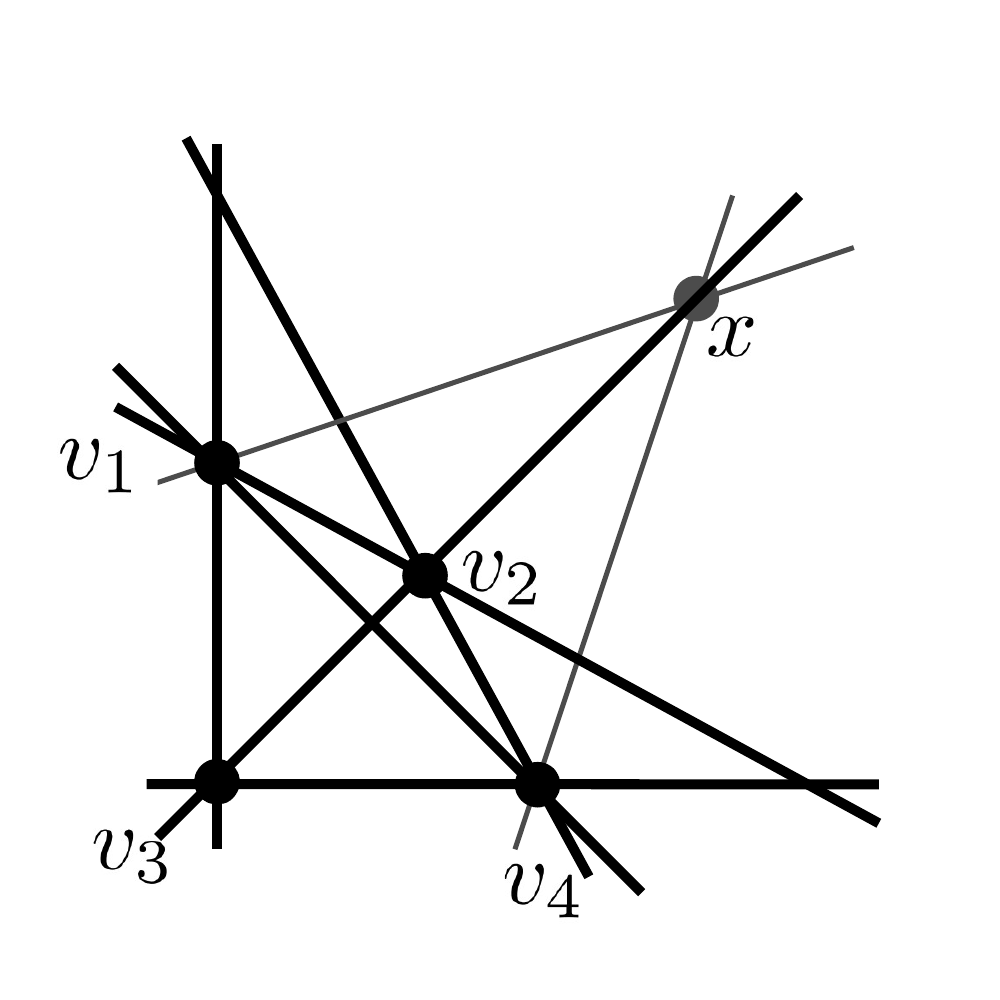}}
\qquad
\subfigure[Type 3]{\includegraphics[scale=0.4]{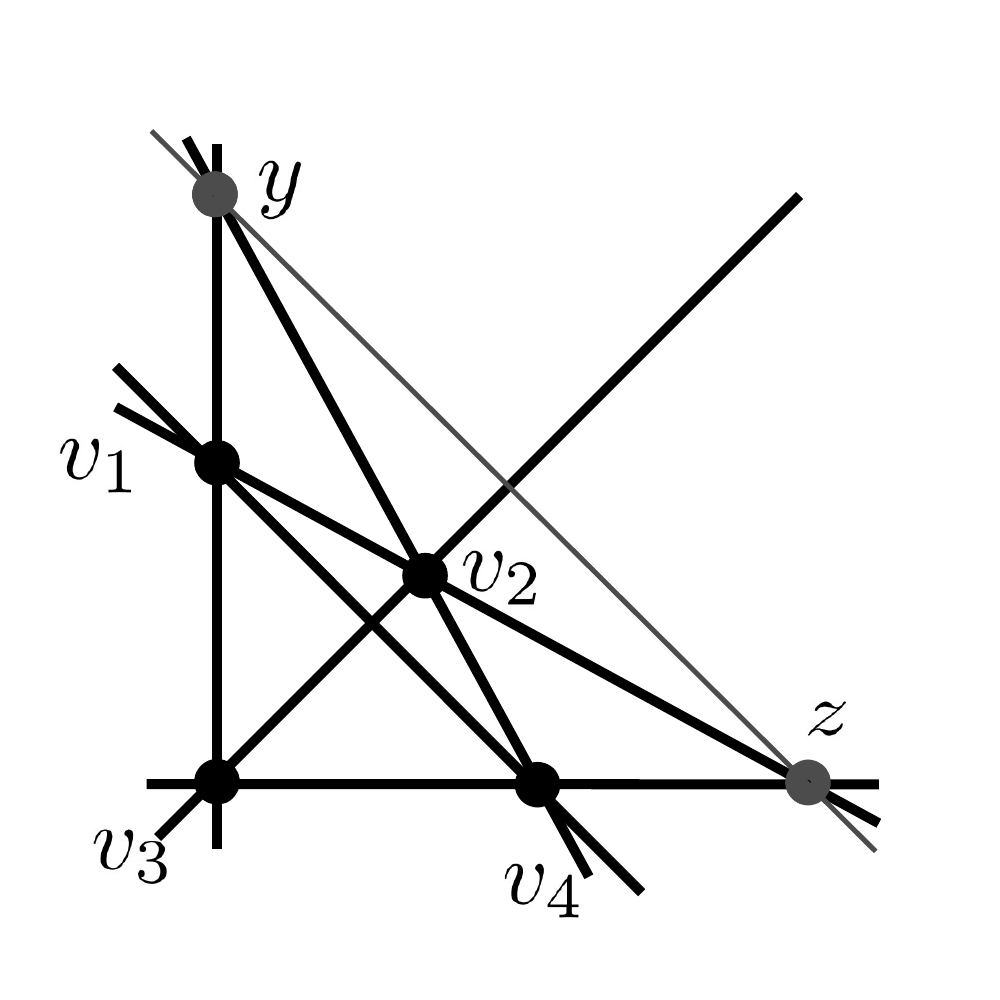}}
\caption{Types of dangerous sets.}
\label{fig:ks2}
\end{center}
\end{figure}

We are now ready to state the main result of this section.

\begin{lem} \label{lem:1}
Let $q$ be a sufficiently large prime and $\alpha = (\log q)^2$. 
Then, there exists a $q$-uniform hypergraph $H=(V,\mathcal{L}')$ of order $q^2$ such that:
\begin{enumerate}[label=(H\arabic*)]
\item\label{H:0} Any two vertices are contained in at most one hyperedge;
\item\label{H:3} For every $v \in V$, $d_H(v) \leq 2 \alpha$; 
\item\label{H:4} $|\mathcal{D}| \leq 2 \alpha^8 q$, where $\mathcal{D}$ is the set of $\mathcal{L}'$-dangerous subsets;
\item\label{H:1} For every $A \in {V \choose q}$, $|\mathcal{L}_A'|\geq \frac{\alpha q}{4}$;
\item\label{H:5} For every integer $1 \leq \gamma \leq \frac{q}{16}$ and every $B \in {V \choose 16 \gamma q}$, $|\mathcal{L}'_{B, \gamma }| \geq \frac{\alpha q}{8}$.
\end{enumerate}
\end{lem}

Before proving the above lemma, we state a basic form of the Chernoff bound (as appearing in Corollary 2.3 of \cite{JLR}) and mention what we will refer to as the union bound.

\begin{c1}
If $X \sim \text{Bi}(n,p)$ and $0 < \varepsilon \leq \frac{3}{2}$, then
$$\Pr \Big( |X - E(X)| \geq \varepsilon \cdot E(X) \Big) \leq 2 \exp \left\{ -\frac{E(X) \varepsilon^2 }{3} \right\} .$$ %\leq  2 \exp \left\{ -\frac{np \varepsilon^2 }{3} \right\}
\end{c1}

\begin{ub}
If $E_i$ are events, then 
$$\Pr \Big( \bigcup_{i=1}^k E_i \Big) \leq k \cdot \max \{\Pr(E_i): i \in [k]\}.$$ %\leq \sum_{i=1}^k{\Pr(Y_i)}
\end{ub}

\begin{proof}[Proof of Lemma~\ref{lem:1}] Take $(V,\mathcal{L})$ to be a hypergraph established by Lemma \ref{lem:0}. Let $\mathbb{H}=(V,\mathcal{L}')$ be a random sub-hypergraph of $(V,\mathcal{L})$ where every line in $\mathcal{L}$ is taken independently with probability 
$$
\frac{\alpha}{q}=\frac{(\log q)^2}{q}.
$$
Since $\mathbb{H}$ is a subgraph of $(V, \mathcal{L})$ any two vertices are in at most one line, so $\mathbb{H}$ always satisfies \ref{H:0}. We will show $\mathbb{H}$ fails to satisfy \ref{H:3} and \ref{H:1} with probability at most $o(1)$ and that $\mathbb{H}$ fails to satisfy \ref{H:4} with probability at most $\frac{1}{2}$. Together this implies $\mathbb{H}$ satisfies \ref{H:0}-\ref{H:1} with probability at least $1-\frac{1}{2}-o(1)$, establishing the existence of a hypergraph $H$ that satisfies \ref{H:0}-\ref{H:1}. Finally, we use a counting argument to show that any such $H$ necessarily satisfies \ref{H:5}.

\emph{\ref{H:3}:} We first show that the probability that there exists a vertex of degree greater than $2 \alpha$ is $o(1)$. Observe for fixed $v \in \mathbb{H}$, $d_{\mathbb{H}}(v) \sim Bi(q, \frac{\alpha}{q})$ and $E(d_{\mathbb{H}}(v))= \alpha.$ So by the Chernoff bound with $\varepsilon = 1$, 
$$ \Pr \Big( d_{\mathbb{H}}(v) \geq 2 \alpha \Big) \leq \Pr \Big( |d_{\mathbb{H}}(v)-\alpha| \geq \alpha \Big) \leq 2  \exp \left\{ -\frac{\alpha}{3} \right\}. $$ 
Thus by the union bound the probability that there exists some $v \in V$ with $d_{\mathbb{H}}(v) \geq 2 \alpha$ is at most
$$q^2 \cdot 2  \exp \left\{-\frac{\alpha}{3}\right\} = 2 \exp \left\{2 \log q-\frac{(\log q)^2}{3}\right\} = o(1).$$

\emph{\ref{H:4}:} In order to show $|\mathcal{D}| > 4 \alpha^8 q$ with probability at most $\frac{1}{2}$, we begin by counting the number of $\mathcal{L}$-dangerous subsets of each type. Clearly the number of Type 1 $\mathcal{L}$-dangerous subsets is at most ${q^2 \choose 5}$. To count the number of Type 2 $\mathcal{L}$-dangerous subsets, first choose $\{v_1,v_2,v_3,v_4\}$ then $x$, observing $x$ must lie on one of the 6 lines which each have at most $q$ vertices. Thus there are at most ${q^2 \choose 4} (6q)$ configurations of this type. To count the number of Type 3 $\mathcal{L}$-dangerous subsets, observe the lines $L(v_i,v_j)$ for $1 \leq i < j \leq 6$ intersect at at most 3 points other than $v_1,v_2,v_3,v_4$. Hence there are at most ${q^2 \choose 4} {3 \choose 2}$ subsets of this type in $\mathcal{L}$.

Since $\mathcal{L}$-dangerous subsets of Type 1, Type 2, and Type 3 have 10, 8, and 7 lines respectively, an $\mathcal{L}$-dangerous subset of each type will be $\mathcal{L}'$-dangerous with respective probabilities $\left(\frac{\alpha}{q}\right)^{10}, \left(\frac{\alpha}{q}\right)^8$, and $\left(\frac{\alpha}{q}\right)^7$. By the linearity of expectation, we now compute 
$$ E(|\mathcal{D}|) \leq {q^2 \choose 5} \cdot \left( \frac{\alpha}{q} \right)^{10}+ {q^2 \choose 4} (6q) \cdot \left( \frac{\alpha}{q} \right)^{8} + {q^2 \choose 4} {3 \choose 2} \cdot \left( \frac{\alpha}{q} \right)^{7} \le \alpha^{10} + \frac{q \alpha^8}{4} + \frac{q \alpha^7}{8}  \leq q \alpha^8. $$

Thus, the Markov inequality yields,
$$\Pr \Big( |\mathcal{D}| \ge 2 \alpha^8 q \Big) \le \Pr \Big( |\mathcal{D}| \ge 2 E ( |\mathcal{D}|) \Big) \le \frac{1}{2}.$$

\emph{\ref{H:1}:} We will now prove that the probability that there exists $A \in {V \choose q}$ such that $|\mathcal{L}_A'| < \frac{\alpha q}{4}$ is $o(1)$. Begin by considering any fixed $A \in {V \choose q}$. Then by Lemma \ref{lem:0}, $|\mathcal{L}_A| \geq \frac{q^2}{2}$, so we may fix $X \subseteq \mathcal{L}_A$ with $|X|=\frac{q^2}{2}$. Let $X'= X \cap \mathcal{L}'$. Since each line in $X$ appears in $X'$ independently with probability $\frac{\alpha}{q}$, $|X'| \sim Bi(\frac{q^2}{2},\frac{\alpha}{q})$ and $E(|X'|)=\frac{\alpha q}{2}$. Hence by the Chernoff bound with $\varepsilon=\frac{1}{2}$, 
$$\Pr \Big( |\mathcal{L}_A'| < \frac{\alpha q}{4} \Big) \leq \Pr  \Big( |X'|< \frac{\alpha q}{4} \Big) \leq \Pr \Big( \Big| |X'|-\frac{\alpha q}{2} \Big|  \geq  \frac{\alpha q}{4} \Big) \leq 2\exp\left\{ - \frac{\alpha q}{24} \right\}.$$
Consequently by the union bound, the probability that there exits some $A \subseteq V$, $|A|=q$, with $|\mathcal{L}_A'| < \frac{\alpha q}{4} $ is at most

$$ {q^2 \choose q} \cdot 2\exp\left\{- \frac{\alpha q}{24} \right\} \leq q^{2q} \cdot 2\exp\left\{ - \frac{(\log q)^2 q}{24} \right\} = 2 \exp\left\{2q \log q - \frac{q(\log q)^2}{24} \right\}=o(1).$$

\emph{\ref{H:5}:} Finally, we will establish the following deterministic property: If $H$ satisfies \ref{H:3} and \ref{H:1}, then $H$ also satisfies \ref{H:5}.

Consider arbitrary fixed $0 \leq \gamma \leq \frac{q}{16}$ and $B \in {V \choose 16 \gamma q}$. Let $B=B_1 \cup B_2 \cup \dots \cup B_{16 \gamma}$ be a partition of $B$ into $16 \gamma$ sets of size $q$. Consider the auxiliary bipartite graph \emph{Aux} with bipartition $\{B_1, B_2, \dots, B_{16 \gamma}\} \cup \mathcal{L}'$. Join $B_i$ to $L \in \mathcal{L}'$ if $B_i \cap L \not = \emptyset$. By property \ref{H:1} $d_{Aux}(B_i) \geq \frac{\alpha q}{4}$ for all $i \in [16 \gamma]$, and thus the number of edges in \emph{Aux} satisfies 
\begin{equation}\label{eq:H5a}
|e(Aux)| \geq \frac{\alpha q}{4} |\{B_1, B_2, \dots, B_{16 \gamma}\}|=4 \alpha q \gamma.
\end{equation}
 On the other hand, clearly $d_{Aux}(L) \leq |\{B_1, B_2, \dots, B_{16 \gamma}\}| = 16 \gamma$ for all $L \in \mathcal{L}'$ and by definition $d_{Aux}(L') \leq \gamma$ for all $L' \in \{ \mathcal{L}' \setminus \mathcal{L}_{B, \gamma }' \}$. Also keeping in mind that by \ref{H:3} \newline
$| \mathcal{L}' \setminus \mathcal{L}_{B, \gamma }' | \leq |\mathcal{L}'| = \sum_{v \in V} \frac{d_H(v)}{q} \leq q^2\frac{2 \alpha}{q}=2 \alpha q$, we compute
\begin{equation}\label{eq:H5b}
|e(Aux)|  \leq |\mathcal{L}_{B, \gamma }'| \cdot 16 \gamma + |\{ \mathcal{L}' \setminus \mathcal{L}_{B, \gamma }' \}| \cdot \gamma \leq |\mathcal{L}_{B, \gamma }'| \cdot 16 \gamma + 2 \alpha q \cdot \gamma.
\end{equation}
Comparing (\ref{eq:H5a}) and (\ref{eq:H5b}), we obtain 
$$4 \alpha q \gamma \leq |e(Aux)|  \leq |\mathcal{L}_{B, \gamma }'| \cdot 16 \gamma + 2 \alpha q \gamma,$$
which yields $|\mathcal{L}'_{B, \gamma }| \geq \frac{\alpha q}{8}$.
\end{proof}

%%%%%%%%%%%%%%%%%%%%%%%%%%%%%%%%%%%%%%%%%%%%%%%%%%%%%%%%%%%%
%%%%%%%%%%%%%%%%%%%%  The Graph G %%%%%%%%%%%%%%%%%%%%%%%%%%
%%%%%%%%%%%%%%%%%%%%%%%%%%%%%%%%%%%%%%%%%%%%%%%%%%%%%%%%%%%%

\section{The Graph $G$}\label{sec:G}
Based upon the hypergraph $H$ established in the previous section, we will construct a graph $G$ with the following properties. 
\begin{lem} \label{lem:3}
Let $q$ be a sufficiently large prime, $\alpha = (\log q)^2$,  $\beta=(\log q)^{4s^2}$, and $s\ge 3$. Then, there exists a graph $G=(V,E)$ of order $q^2$ such that:
\begin{enumerate}[label=(G\arabic*)]
\item\label{G:1} For every $C \in {V \choose 16 s q}$, $G[C]$ contains a copy of $K_s$;
\item\label{G:2} For every $U \in {V \choose 64 s \beta q}$, $G[U]$ contains $\frac{\alpha \beta^2 q}{16}$ edge disjoint copies of~$K_s$;
\item\label{G:3} Every edge $xy \in E$ is in at most $6^s \alpha^{2s-2}$ copies of $K_{s+1}$;
\item\label{G:4} If $s\ge 4$, then $G$ can be made $K_{s+2}$ free by removing $2 \alpha^8 q $ vertices.
\end{enumerate}
\end{lem}

\begin{proof}  Fix a hypergraph $H=(V,\mathcal{L}')$ as established by Lemma~\ref{lem:1}. Construct the random graph $\mathbb{G}$ as follows. For every $L\in \mathcal{L}'$, let $\chi_L : L \to [s]$ be a random partition of the vertices of $L$ into $s$ classes, where for every $v \in L$, a class $\chi_L(v) \in [s]$ is assigned uniformly and independently at random. Then, let $xy \in E$ if $\{x,y\} \subseteq L$ for some $L \in \mathcal{L}'$ and $\chi_L(x) \not = \chi_L(y)$. Thus for every $L \in \mathcal{L}'$, $\mathbb{G}[L]$ is a complete $s$-partite graph with vertex partition $L = \chi^{-1}_L (1) \cup \chi^{-1}_L (2) \cup \dots \cup \chi^{-1}_L (s)$ (where the classes need not have the same size and the unlikely event that a class is empty is permitted). Observe that not only are $G_{H'}[L]$ and $G_{H'}[L']$ are edge disjoint for distinct $L, L' \in \mathcal{L}'$, but also that the partitions for $L$ and $L'$ were determined independently.

We will show $\mathbb{G}$ does not satisfy \ref{G:1} and \ref{G:2} with probability at most $o(1)$ and that $\mathbb{G}$ always satisfies \ref{G:3} and \ref{G:4}. Hence the probability that $\mathbb{G}$ satisfies properties \ref{G:1}-\ref{G:4} is at least $1-o(1)$, implying the existence of a graph $G$ described in the lemma.

\emph{\ref{G:1}:} Consider any $C \in {V \choose 16sq}$. We will bound the probability that $G[C] \not \supset K_s$. By \ref{H:5} with $\gamma = s$, the set of lines $\mathcal{L}'_{C,s}$ that intersect $C$ in at least $s$ vertices has cardinality $|\mathcal{L}'_{C,s}| \geq \frac{\alpha q}{8}$. For each $L \in \mathcal{L}'_{C,s}$, let $X_L$ be the event $K_s \not \subseteq \mathbb{G}[L \cap C]$. Since $|L \cap C| \geq s$ for all $L \in \mathcal{L}'_{C,s}$, $\Pr (X_L) \leq 1- \frac{s!}{s^s}$. By independence, 
$$\Pr \Big( K_s \not \in \mathbb{G}[C] \Big) \leq \Pr \Big( \bigcap_{L \in \mathcal{L}'_{C,s}} X_L \Big) \leq {\left( 1- \frac{s!}{s^s} \right) }^{|\mathcal{L}'_{C,s}|} \leq {\left( 1- \frac{s!}{s^s} \right) }^\frac{\alpha q}{8}\leq \exp { \left\{ - \frac{s!}{s^s}\frac{\alpha q }{8} \right\} }.$$
So by the union bound, the probability that there exists a subset of $16 s q$ vertices in $\mathbb{G}$ that contains no $K_s$ is at most
$$ 
{q^2 \choose 16 s q} \exp { \left\{ - \frac{s!}{s^s}\frac{\alpha q }{8} \right\} }
\le  q^{16 s q} \exp { \left\{ - \frac{s!}{s^s}\frac{\alpha q }{8} \right\} }
= \exp\left\{16 s q \log q - \frac{s! q (\log q)^2}{8s^s} \right\} = o(1).
$$

\emph{\ref{G:2}:} For arbitrary $U \in {V \choose 64 s \beta q}$, we will bound the probability that $G[U]$ does not contain $\frac{\alpha \beta^2 q}{16}$ edge disjoint copies of $K_s$. By \ref{H:5} with $\gamma = 4 s \beta$, we may fix a subset $\mathcal{Z}_U \subseteq \mathcal{L}_{U,4 s \beta}'$ of exactly $\frac{\alpha q}{8}$ lines with the property that each line has intersection at least $4 s \beta$ with $U$. We will consider the lines in $\mathcal{Z}_U$ that contain the complete balanced $s$-partite graph on $2s \beta$ vertices, which we denote by $K_{2\beta,\dots,2\beta}$. Define $\mathcal{Z}_U' = \{L \in \mathcal{Z}_U : K_{2\beta,\dots,2\beta} \subseteq \mathbb{G}[L \cap U] \}$. The graph $K_{2\beta,\dots,2\beta}$ certainly contains at least $\beta^2$ edge disjoint $K_s$ (Since we may choose a prime $\beta \leq p \leq 2\beta$ and it follows from \cite{ACD} that we may then decompose $K_{p,\dots,p}$ into $p^2$ edge disjoint copies of $K_s$; this suffices for our purposes, but stronger results are know). Thus if we show $|\mathcal{Z}_U'| \geq \frac{\alpha q}{16}$ it will imply that $\mathbb{G}[U]$ contains at least $|\mathcal{Z}_U'| \cdot \beta^2 \geq \frac{\alpha \beta^2 q}{16}$ edge disjoint copies of $K_s$.

For $L \in \mathcal{Z}_U$, let $Y_L$ be the event that $L \not \in \mathcal{Z}_U'$ and fix $L_{4 s \beta} \subseteq L \cap U, |L_{4 s \beta}|=4s \beta$. Now $Y_L$ will occur only if $|\chi_{L}^{-1}(i) \cap L_{4 s \beta}| < 2 \beta$ for some $i \in [s]$. Defining $X_i = |\chi_{L}^{-1}(i) \cap L_{4 s \beta}|$, observe $X_i  \sim Bi(4s\beta, \frac{1}{s})$ and $E(X_i) = 4 \beta$. Chernoff's inequality reveals 
$$\Pr \Big( X_i < 2 \beta \Big) \leq \Pr \Big( |X_i - E(X_i)| \geq \frac{E(X_i)}{2} \Big) \leq 2 \exp\left\{-\frac{4 \beta }{12}\right\}=2 \exp\left\{-\frac{\beta }{3}\right\}.$$
By the union bound, $\Pr ( Y_L ) \leq \Pr \Big( \bigcup_{i \in s} (X_i \leq 2 \beta)  \Big) \leq s \cdot 2 \exp\left\{-\frac{ \beta }{3}\right\}.$

By independence, the probability that $Y_L$ occurs for at least $\frac{\alpha q}{16}=\frac{|\mathcal{Z}_U|}{2}$ of the lines in $\mathcal{Z}_U$ is at most

$$\binom{|\mathcal{Z}_U|}{\frac{|\mathcal{Z}_U|}{2}} \left( 2s \exp\left\{-\frac{ \beta}{3}\right\} \right)^{|\mathcal{Z}_U| / 2} 
\le 4^{|\mathcal{Z}_U| / 2} \left( 2s \exp\left\{-\frac{ \beta}{3}\right\} \right)^{|\mathcal{Z}_U| / 2}  = \left( 8s \exp\left\{-\frac{ \beta}{3}\right\} \right)^{\frac{\alpha q}{16}}.$$

That is, we have shown $|\mathcal{Z}_U'| < \frac{\alpha q}{16}$ with probability at most $\left( 8s \exp\left\{-\frac{\beta}{3}\right\} \right)^{\frac{\alpha q}{16}}$ for fixed~$U$. Thus by the union bound, the probability that there exits some $U \subseteq V$ with $|U|= 64 s \beta q$ such that $|\mathcal{Z}_U'| < \frac{\alpha q}{16}$  is at most

\begin{align*}
{{q^2} \choose {64 s \beta q}} \left( 8s \exp\left\{ - \frac{\beta}{3}  \right\} \right)^{\frac{\alpha q}{16}} 
&\le q^{64s \beta q } ( 8s )^{\frac{\alpha q}{16}} \left( \exp\left\{ - \frac{\beta}{3}  \right\} \right)^{\frac{\alpha q}{16}} \\
&\leq  \exp \left\{64 s \beta q \log q + \frac{\alpha q}{16} \log(8s) -  \frac{\alpha \beta q }{48} \right\}=o(1).
\end{align*}

\emph{\ref{G:3}:} For any $xy \in E$, we will show the number of copies of $K_{s+1}$ that contain $xy$ is at most $6^s \alpha^{2s-2}$.
\begin{figure}[h!]
\centering
\includegraphics[scale=0.5]{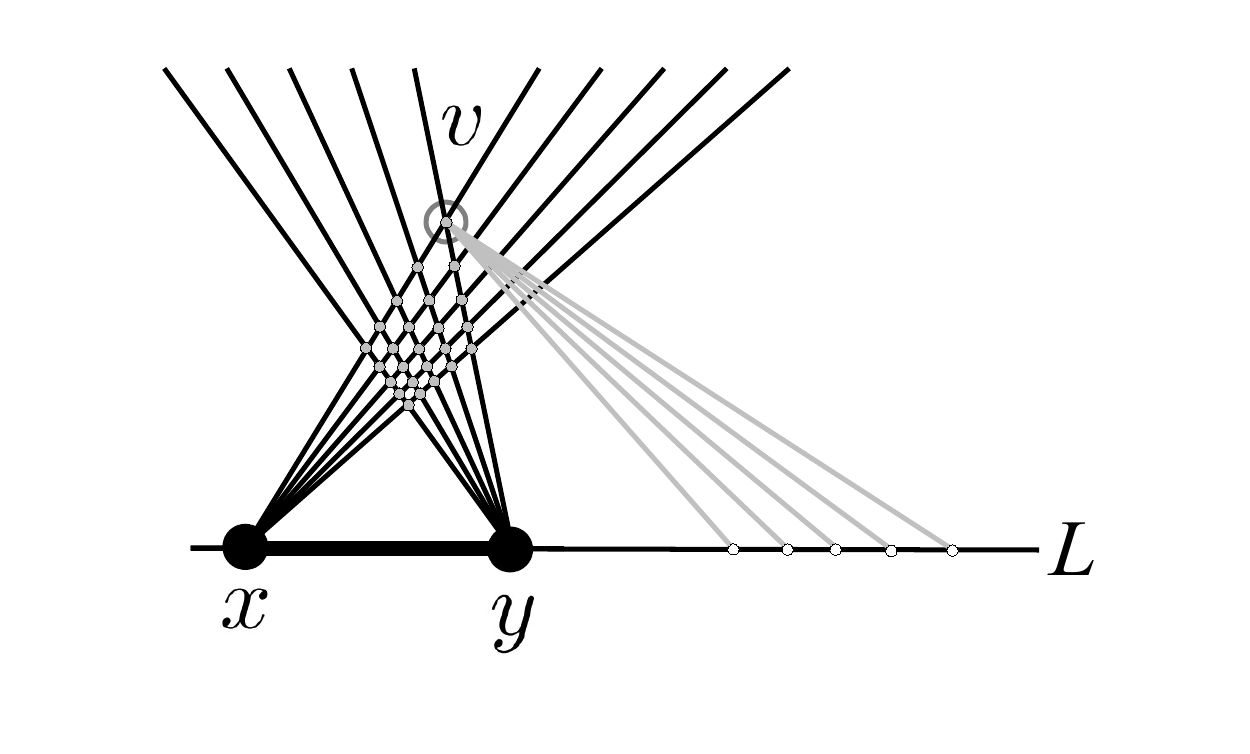}
\caption{Counting $K_{s+1}$ in $\mathbb{G}$ that contains a fixed edge $xy$ by considering lines in $H$.}
\label{fig:ks1}
\end{figure}
Let $L \in \mathcal{L}'$ be the unique line such that $\{x,y\} \subseteq L$ as depicted in Figure~\ref{fig:ks1}. Let $N = (N_H(x) \cap N_H(y)) \setminus L$ be the set of all vertices not on $L$ that are collinear with both $x$ and $y$. Since $d_H(x) , d_H(y) \leq  2\alpha$ by \ref{H:3}, we infer that $|N| \leq 4{\alpha}^2$. Because $K_{s+1} \not \subseteq G[L]$, if a $K_{s+1}$ is to contain $x$ and $y$ it must contain at least one vertex $v \in N$. There are at most $|N| \le 4{\alpha}^2$ choices for this vertex $v$. Once $v$ has been chosen, each of the remaining $s-2$ vertices of the $K_{s+1}$ must lie in $N$ or in $L \cap N_H(v)$. Since $|N|+ |L \cap N_H(v) | \leq 4{\alpha}^2+2 \alpha$, the number of $K_{s+1}$ containing the edge $xy$ is at most ${4 \alpha}^2 (4{\alpha}^2+2\alpha)^{s-2} \leq 6^s \alpha^{2s-2}$.

\emph{\ref{G:4}:} We will finally show that if $s \geq 4$, $G$ can be made $K_{s+2}$ free be removing at most $2 \alpha^8 q $ vertices. By \ref{H:4}, all $\mathcal{L}'$-dangerous sets can be destroyed by removing $2 \alpha^8 q$ vertices, so it suffices to shown that every $K_{s+2}$ in $\mathbb{G}$ contains a $\mathcal{L}'$-dangerous subset.

Let $K$ be any copy of $K_{s+2}$ in $\mathbb{G}$. By assumption $s \geq 4$, so $K$ must have at least 6 vertices, which clearly form a $\mathcal{L}'$-complete set.

We first show that $K$ contains 4 vertices in general position. Suppose otherwise. Then there is some line $L \in \mathcal{L}'$ that contains 3 vertices $\{p_1,p_2,p_3\}$ of $K$. Since $K_{s+1} \not \subseteq \mathbb{G}[L]$, there must exist two vertices $a$ and $b$ in $K \setminus L$. Observe $\{a,b\}$ and any 2 vertices in $\{p_1,p_2,p_3\} \setminus L(a,b)$ are in general position.

Now fix 4 vertices $\{v_1,v_2,v_3,v_4\}$ of $K$ that are in general position and let $u_1,u_2$ be any two other vertices of $K$. Three cases are now considered. If either $u_1$ or $u_2$ do not lie on any of the 6 lines $L(v_i,v_j)$ for $1 \leq i < j \leq 4$, then there is a $\mathcal{L}'$-dangerous subset of Type~1. If either $u_1$ or $u_2$ lie on exactly one line in $L(v_i,v_j)$ for $1 \leq i < j \leq 4$, then there is a $\mathcal{L}'$-dangerous subset of Type~2. In the remaining case where both $u_1$ and $u_2$ each lie on at least 2 lines in $L(v_i,v_j)$ for $1 \leq i < j \leq 4$, then there is a $\mathcal{L}'$-dangerous subset of Type~3.
\end{proof}

%%%%%%%%%%%%%%%%%%%%%%%%%%%%%%%%%%%%%%%%%%%%%%%%%%%%%%%%%%%%
%%%%%%%%%%%%%%%%%%%%  Proof of Theorem  %%%%%%%%%%%%%%%%%%%%
%%%%%%%%%%%%%%%%%%%%%%%%%%%%%%%%%%%%%%%%%%%%%%%%%%%%%%%%%%%%

\section{Proof of Theorem \ref{thm:s1} and \ref{thm:s2}}\label{sec:proofs}

Consider any sufficiently large integer $n$ and $s \geq 3$. By Bertrand's postulate, we can find a prime $q$ such that $4n \leq q^2 \leq 16n$. Fix a graph $G$ procured by Lemma \ref{lem:3} of order $q^2$ and as before take
$$
\alpha = (\log q)^{2} \hspace{1cm} \text{and} \hspace{1cm} \beta = (\log q)^{4s^2}.
$$
Theorem~\ref{thm:s1} and Theorem~\ref{thm:s2} are now proved by considering different subgraphs of $G$ of order $n$.

\begin{proof}[Proof of Theorem \ref{thm:s2}]
Consider the case where $s \geq 4$. To prove the theorem, we will show there exists a $K_{s+2}$-free induced subgraph of $G$ of order $n$ with the property that every subset of order $64 s \sqrt{n}$ contains a copy of $K_s$.

By \ref{G:1}, every set of size $16  s  q $ in $G$ contains $K_s$, so certainly every subset of size $64 s \sqrt{n} \geq 16  s  q$ in any induced subgraph of $G$ must also contain a copy of $K_s$.  Thus it will suffice to show that there is a $K_{s+2}$-free subset of $G$ of order $n$. But by \ref{G:4}, we know that there is a set $R \subseteq V(G)$ of size $|R|=2 \alpha^8 q \leq n$ such that $G [V \setminus R]$ will be $K_{s+2}$-free. Finally since $|V \setminus R | \geq 4n - n \geq n$, the induced graph of $G$ on any $n$ vertices in $V \setminus R$ will have the desired properties.
\end{proof}

\begin{proof}[Proof of Theorem \ref{thm:s1}]
For $s \geq 3$, we will concentrate on constructing a $K_{s+1}$-free graph $G'$ on $q^2$ vertices with the property that every subset of size $64s \beta q$ vertices contains a copy of $K_s$. Since $\log(4n)\le 2\log n $,
$$
64s \beta q = 64 s (\log q)^{4s^2} q \leq 
64 s (\log 4n)^{4s^2} 4n \leq
2^{4s^2+8} (\log n)^{4s^2} n,
$$ 
and so any induced subgraph of $G'$ of order $n$ will also be $K_{s+1}$-free and have the property that every set of order $2^{4s^2+8} (\log n)^{4s^2} n$ contains a copy of $K_s$, exactly as desired.

Let $G'$ be a random subgraph of $G$ where each edge is taken with probability 
$$
\frac{1}{\gamma}, \quad \text{where}\ \gamma = (\log q)^{8}.
$$ 
For a set $S \in {V(G) \choose s+1}$ that spans a copy of $K_{s+1}$ in $G$, let $A_S$ to be the event that all the edges of $S$ are in $G'$. Hence, $\bigcap{\overline{A_S}} $ means that $K_{s+1} \not \subseteq G'$. For a set $U \in {V(G) \choose 64 s \beta q}$ let $\mathcal{K}_U$ be a (fixed) set of 
$$
m=\frac{1}{16}\alpha\beta^2q
$$ 
edge disjoint copies $K_s$ contained in $U$, which are known to exist by \ref{G:2}. Define $B_U$ to be the event that none of the $m$ edge disjoint $K_s$ appear in $G'$. Hence, $ \bigcap \overline{B_U} $ implies that for every $U \in {V(G) \choose 64 s \beta q}$ one of the disjoint copies of $K_s$ in $G[U]$ appears in $G'$. It will suffice to show that the probability that $\left(\bigcap{\overline{A_S}} \right) \cap  \left( \bigcap \overline{B_U} \right)$ occurs is nonzero. In order to show this, we apply the Local Lemma (see, e.g., Lemma 5.1.1 in~\cite{AS}). 

\begin{lovasz}
Let $E_1, E_2,\dots, E_k$ be events in an arbitrary probability space. A directed graph $D$ on the set of vertices $\{1,2,\dots,k\}$ is called a {\em dependency digraph} for the events $E_1, E_2,\dots, E_k$ if for each $i$, $1\le i \le k$, the event $E_i$ is mutually independent of all the events $\{E_j : (i, j) \not \in D\}$. Suppose that $D$ is a dependency digraph for the above
events and suppose there are real numbers $z_1,\dots, z_k$ such that $0 \le z_i < 1$ and
$Pr(E_i) \le z_i \prod_{(i,j)\in D} (1-z_j)$ for all $1\le i\le k$. Then, 
$\Pr\left(\bigcap_{i=1}^k {\overline{E_i}} \right) > 0$.
\end{lovasz}

Let $D$ be a dependency graph that corresponds to all events $A_S$ and $B_U$.
Observe that $A_S$ depends only on the ${s+1 \choose 2}$ edges in $S$ and $B_U$ depends only on the $m {s \choose 2}$ edges of the $K_s$ in $\mathcal{K}_U$. Also, observe that the number of events of the type $B_U$ is ${q^2 \choose 64 s \beta q} \leq q^{64 s \beta q}$. Thus by~\ref{G:3}, a fixed event $A_S$ depends on at most 
$$
d_{AA} = {s+1 \choose 2}  6^s \alpha^{2s-2} 
$$
other events $A_{S'}$ and at most
$$
d_{AB} = q^{64 s \beta q} 
$$ %{s+1 \choose 2} n^{320 \beta \sqrt{n}}
events $B_U$. Similarly, a fixed event $B_U$ depends on at most 
$$
d_{BA} = m {s \choose 2}  6^s \alpha^{2s-2}
$$
events $A_S$ and at most
$$
d_{BB} = q^{64 s \beta q} 
$$ % m {s \choose 2}  n^{320 \beta \sqrt{n}} 
other events $B_{U'}$. Let 
$$
x= \frac{1}{\alpha^{2s^2}}\quad \text{and} \quad y= \frac{1}{(\log q)^{4s^2} q^{64 s \beta q}}.
$$
To finish the proof, due to the Local Lemma it suffices to show that 
\begin{equation}\label{eq:lllx}
{\left( \frac{1}{\gamma} \right)}^{s+1 \choose 2} \leq x(1-x)^{d_{AA}}(1-y)^{d_{AB}},
\end{equation}
\begin{equation}\label{eq:llly}
\left(1-\left( \frac{1}{\gamma} \right)^{s \choose 2}\right)^{m} \leq y (1-x)^{d_{BA}}(1-y)^{d_{BB}}.
\end{equation}

First we show that \eqref{eq:lllx} holds. Using the fact that $e^{-2x} \leq 1-x$ for $x$ sufficiently small (observe that $x \to 0$ with $q \to \infty$), a sufficient condition for \eqref{eq:lllx} will be
$$
{\left( \frac{1}{\gamma} \right)}^{s+1 \choose 2} \leq x\, e^{-2xd_{AA}}\, e^{-2yd_{AB}},
$$
and equivalently,
$$
{s+1 \choose 2} \log \left( \gamma \right) \geq \log \left( \frac{1}{x} \right)+2xd_{AA}+2yd_{AB}.
$$
The latter immediately follows from the following three inequalities (which can be easily verified):
\begin{align*}
&\frac{2s^2}{2s^2+2s}{s+1 \choose 2} \log \left( \gamma \right) \geq \log \left( \frac{1}{x} \right),\\ 
&\frac{s}{2s^2+2s} {s+1 \choose 2} \log \left( \gamma \right) \geq 2xd_{AA},\\
&\frac{s}{2s^2+2s} {s+1 \choose 2} \log \left( \gamma \right) \geq 2yd_{AB}.
\end{align*}

Similarly, using the facts that $e^{-2y} \leq 1-y$ for $y$ sufficiently small and that  $1-\left( \frac{1}{\gamma} \right)^{s \choose 2} \leq e^{-\left( \frac{1}{\gamma} \right)^{s \choose 2}}$, \eqref{eq:llly} will be satisfied if
$$
e^{-m {\left( \frac{1}{\gamma} \right)}^{s \choose 2}}  \le  y\, e^{-2xd_{BA}}\, e^{-2yd_{BB}},
$$
and equivalently,
$$
m {\left( \frac{1}{\gamma} \right)}^{s \choose 2}  \geq  \log {\left( \frac{1}{y} \right) } +2xd_{BA} +2yd_{BB}.
$$
As before the latter will follow from the following easy to check inequalities: 
\begin{align*}
&\frac{1}{3} m {\left( \frac{1}{\gamma} \right)}^{s \choose 2}  \geq \log  {\left( \frac{1}{y} \right) },\\
&\frac{1}{3} m {\left( \frac{1}{\gamma} \right)}^{s \choose 2}  \geq  2xd_{BA},\\
&\frac{1}{3} m {\left( \frac{1}{\gamma} \right)}^{s \choose 2}  \geq  2yd_{BB}.
\end{align*}
This completes the proof of Theorem~\ref{thm:s1}.
\end{proof}

%%%%%%%%%%%%%%%%%%%%%%%%%%%%%%%%%%%%%%%%%%%%%%%%%%%%%%%%%%%%
%%%%%%%%%%%%%%%%%%%%  Proof of Theorem  %%%%%%%%%%%%%%%%%%%%
%%%%%%%%%%%%%%%%%%%%%%%%%%%%%%%%%%%%%%%%%%%%%%%%%%%%%%%%%%%%

\section{Concluding Remarks}

We close this paper by discussing how the asymptotic behavior of $f_{s,t}(n)$ changes for different values of $3 \le s <t$. 

If the difference between $s$ and $t$ is fixed, we make the following observation based upon the lower bound in Sudakov~\cite{SU} (and Fact 3.5 in~\cite{DR}) and Corollary \ref{cor3}.
\begin{obs}
For any $\varepsilon > 0$ and an integer $k\ge 2$ there is a constant $s_0=s_0(k, \varepsilon)$ such that for all $s\ge s_0$,
$$
\Omega\big{(}n^{\frac{1}{2} - \varepsilon}\big{)} = f_{s,s+k}(n) = 
O(\sqrt{n}).
$$
\end{obs}
\noindent
In view of this observation and Theorem~\ref{thm:s2} we ask the following.
\begin{que}
For any $s\ge 3$, is $f_{s,s+2}(n) =  o(\sqrt{n})$?
\end{que}

Another interesting question results from fixing that ratio between $s$ and $t$. The following is based upon \cite{SU} and \cite{KR2} respectively. 
\begin{obs}\label{obs:2}
For any $\varepsilon > 0$ and $\lambda \geq 2$ there is a constant $ s_0=s_0(\lambda, \varepsilon)$ such that for all $s\ge s_0$,
$$\Omega\big{(}n^{\frac{1}{2\lambda} - \varepsilon}\big{)} = f_{s,\lfloor \lambda s \rfloor}(n) = O\big{(}n^{\frac{1}{\lambda}}\big{)}.$$
\end{obs}
\noindent
In particular, when $\lambda =3$, we see $\Omega(n^{1/6 - \varepsilon}) = f_{s,\lfloor \lambda s \rfloor}(n) = O(n^{1/3}) $.
\begin{que}
What is the asymptotic behavior of $ f_{s,\lfloor \lambda s \rfloor}(n)$?
\end{que}

Recall that Erd\H{o}s~\cite{ER} asked if for fixed $s+2\le t$, $\lim_{n\to\infty} \frac{f_{s+1,t}(n)}{f_{s,t}(n)} = \infty$. We ask a similar question, that if answered in the affirmative would imply an affirmative answer to the question of Erd\H{o}s.
\begin{que}
For all $t>s\ge 3$, is $\lim_{n\to\infty} \frac{f_{s+1,t+1}(n)}{f_{s,t}(n)} = \infty$?
\end{que}

%%%%%%%%%%%%%%%%%%%%%%%%%%%%%%%%%%%%%%%%%%%%%%%%%%%%%%%%%%%%
%%%%%%%%%%%%%%%%%%%%  BIBLIOGRAPHY  %%%%%%%%%%%%%%%%%%%%%%%%
%%%%%%%%%%%%%%%%%%%%%%%%%%%%%%%%%%%%%%%%%%%%%%%%%%%%%%%%%%%%


\begin{thebibliography}{99}

\bibitem{ACD} R.~Abel, C.~Colbourn, and J.~Dinitz.
{\em Mutually orthogonal Latin squares (MOLS)},
In C.~Colbourn and J.~Dinitz, eds.,
Handbook of Combinatorial Designs, chap. III.3, pp. 160--193. Chapman \& Hall, 2nd ed., 2007.

\bibitem{AK} N.~Alon and M.~Krivelevich,
{\em Constructive bounds for a {R}amsey-type problem},
Graphs Combin. \textbf{13} (1997), 217--225.

\bibitem{AS} N.~Alon and J.~Spencer, 
{\em The Probabilistic Method}, third ed., John Wiley \& Sons Inc., 2008.

\bibitem{BH} B.~Bollob\'as and H.~Hind,
{\em Graphs without large triangle free subgraphs},
Discrete Math. \textbf{87} (1991), no.~2, 119--131. 

\bibitem{CA} P.~Cameron,
{\em Combinatorics:  Topics, Techniques, Algorithms},
Cambridge University Press, Cambridge, 1994.

\bibitem{DM} A.~Dudek and D.~Mubayi,
{\em On generalized Ramsey numbers for 3-uniform hypergraphs},
submitted.

\bibitem{DR} A.~Dudek and V.~R\"odl,
{\em On $K_s$-free subgraphs in $K_{s+k}$-free graphs and vertex Folkman numbers}, 
Combinatorica \textbf{31} (2011), 39--53.

\bibitem{DR2} A.~Dudek and V.~R\"odl,
{\em On the function of Erd\H{o}s and Rogers}, 
in Ramsey Theory: Yesterday, Today and Tomorrow, edited by A.~Soifer, Progress in Mathematics, vol. 285, Springer-Birkh\"auser, 2010, pp. 63--76.

\bibitem{ER} P.~Erd\H{o}s, 
{\em Some of my recent problems in combinatorial number theory, geometry and combinatorics},
in: Graph theory, combinatorics, and algorithms, Vol. 1, 2, Wiley, New York, 1995, 335--349.

\bibitem{ERog} P.~Erd\H{o}s and C.~Rogers,
{\em The construction of certain graphs},
Canad. J.~Math.~\textbf{14} (1962), 702--707.

\bibitem{JLR}  S.~Janson, T.~{\L}uczak, and A.~Ruci\'nski, 
{\em Random Graphs}, Wiley, New York, 2000.  

\bibitem{KR2} M.~Krivelevich,
{\em Bounding {R}amsey numbers through large deviation inequalities},
Random Structures Algorithms \textbf{7} (1995), 145--155.

\bibitem{KR} M.~Krivelevich,
{\em $K^s$-free graphs without large $K^{r}$-free subgraphs},
Combin. Probab. Comput.~\textbf{3} (1994), 349--354.

\bibitem{SH} J.~B.~Shearer, 
{\em On the independence number of sparse gaphs}, 
Random Structures Algorithms \textbf{7} (1995), 269-271.

\bibitem{SU} B.~Sudakov, 
{\em Large $K_r$-free subgraphs in $K_s$-free graphs and some other Ramsey-type problems}, 
Random Structures Algorithms \textbf{26} (2005), 253-265.

\bibitem{SU2} B.~Sudakov, 
{\em A new lower bound for a {R}amsey-type problem}, 
Combinatorica \textbf{25} (2005), 487--498.

\bibitem{Wo} G.~Wolfovitz,
{\em $K_4$-free graphs without large induced triangle-free subgraphs},
to appear in Combinatorica.

\end{thebibliography}
\end{document}